\documentclass[11pt]{amsart}
\usepackage[utf8]{inputenc} % PERMITE USAR DIGITA?O COMO NO WORD
\usepackage[pctex32]{graphics}
\usepackage{amsfonts}
\usepackage{amssymb,amscd,latexsym}
\usepackage{amsmath}
\usepackage{epsfig}
\usepackage{color}

\newtheorem{thm}{Theorem}[section]
\newtheorem{cor}[thm]{Corollary}

\newtheorem{defin}[thm]{Definition}
\newtheorem{Definition}[thm]{Definition}

\newtheorem{rmk}[thm]{Remark}

\newtheorem{conjecture}[thm]{Conjecture}

\newcommand{\codim}{\operatorname{codim}}

\newcommand{\chr}{\operatorname{char}}

\def\C{\mathbb{C}}

\def\K{\mathbb{K}}
\def\rk{\operatorname{rk}}

\newcommand{\Hilb}{\operatorname{Hilb}}

\newcommand{\ann}{\operatorname{Ann}}
\newcommand{\Hom}{\operatorname{Hom}}

\newcommand{\rank}{\operatorname{rank}}

\begin{document}

\title{Lefschetz properties for Artinian Gorenstein algebras presented by quadrics}
\author[R. Gondim]{Rodrigo Gondim*}
\address{Universidade Federal Rural de Pernambuco, av. Don Manoel de Medeiros s/n, Dois Irmãos - Recife - PE
52171-900, Brasil}
\email{rodrigo.gondim@ufrpe.br}
\author[G. Zappalà]{Giuseppe Zappalà}
\address{Dipartimento di Matematica e Informatica, Universit\` a degli Studi di Catania, Viale A. Doria 5, 95125 Catania, Italy}
\email{zappalag@dmi.unict.it}

 \begin{abstract} We introduce a family of Artinian Gorenstein algebras, whose combinatorial structure characterizes the ones presented by quadrics. Under certain hypotheses these
 algebras have non-unimodal Hilbert vector. In particular we provide families of counterexamples to the conjecture that Artinian Gorenstein algebras presented by quadrics should satisfy the weak Lefschetz property. 
 \end{abstract}

\thanks{*Partially supported  by the CAPES postdoctoral fellowship, Proc. BEX 2036/14-2}

\maketitle

\section{Introduction}

It is very useful in algebraic geometry and in commutative algebra to produce geometric or algebraic objects from combinatoric ones. 
Toric varieties (\cite{CLS}) and toric ideals (\cite{Stu}) are the more stablished of these associations. 
We can also cite tropical varieties (\cite{S}), Stanley-Reisner theory (\cite{St,St2}), Artinian algebras given by posets (\cite{HMMNWW}), as incarnations of this 
fruitful interaction among Algebra, Geometry and Combinatorics. 
In this paper we propose a new construction in this direction, associating simplicial complexes to certain Bigraded Artinian Gorenstein algebras. \par
%The Gorenstein hypothesis, that seens to be technical, in the context is the 

The combinatoric structure of the simplicial complex, in our association, characterizes the algebras presented by quadrics. To be more precise, 
it determines the quotient algebra and its ideal on a natural embedding. In particular, the Hilbert vector of such Artinian algebra is determined by the face vector of the complex. 
In this way we construct a concrete family of algebras presented by quadrics and whose Hilbert vector is non unimodal. 
These algebras provide counterexamples for two conjectures posed in \cite{MN1,MN2}.\par

The kind of Algebra we introduce is closely related to Stanley-Reisner theory (\cite{St,St2}). The starting point of both constructions is a homogeneous simplicial complex with 
$m \geq 2$ vertices and dimension $d-1 \geq 1$, associated to a set of square-free monomials in $m$ variables of degree $d$. 
Each monomial represents a facet of the simplicial complex. The very distinct point of view is that our construction is also related to Nagata's idealization (\cite{HMMNWW}).
This point of view is related with the vanishing of Hessian determinant (see \cite{Go,GRu,CRS,MW}). \par

A standard graded $\K$-algebra is said to be presented by quadrics if it is isomorphic to the quotient of a polynomial ring over $\K$ by a homogeneous ideal generated by 
quadratic forms. Also called quadratic algebras, they are related with Koszul algebras and Gr\"obner basis 
(see for instance \cite{Co}). 
From a more geometric point of view quadratic ideals appear as homogeneous ideals of very positive embeddings of any smooth projective varieties. 
As pointed out in \cite{MN2}, Artinian Gorenstein algebras presented by quadrics 
are also related to Eisenbud-Green-Harris conjectures motivated by the Cayley-Bacharach theorem (\cite{EGH}). \par

The Lefschetz properties, on the other side, have attracted a great deal of attention over the years of researchers in  
different subjects including Commutative Algebra, Algebraic Geometry and Combinatorics, see \cite{St,St2,MN1,HMMNWW,HMNW}. The present work lies in the border of these three areas.  \par

In \cite{MN2}, the authors studied Artinian Gorenstein algebras presented by quadrics, they provided some constructions of such algebras and described their possible Hilbert vectors 
in low codimension. 
In \cite{MN1} and \cite{MN2} the authors proposed two conjectures. 

\begin{conjecture} {\bf (Migliore-Nagel injective Conjecture)} \label{conj:WMNC}
 For any Artinian Gorenstein algebra of socle degree at least three, presented by
quadrics, defined over a field $\K$ of characteristic zero there exists $L \in A_1$, such that, the multiplication map $\bullet L:A_1 \to A_2$ is injective.
\end{conjecture}

\begin{conjecture} {\bf (Migliore-Nagel WLP Conjecture)}\label{conj:SMNC}
 Any Artinian Gorenstein algebra presented by quadrics, over a field $\K$ of characteristic zero, has the Weak Lefschetz Property, that is, there exists $L \in A_1$ such that all the maps $\bullet L: A_i \to A_{i+1}$ have maximal rank.
\end{conjecture}

In \cite{MN2} the authors proved the WLP-conjecture for complete intersection of quadratic forms and presented computational evidence for the conjectures in low codimension. 
We want to stress the fact that as soon as the codimension increases with respect to the socle degree surprising phenomena begin to appear. For instance, in codimension $\leq 2$ every 
Artinian algebras have the strong Lefschetz property (see \cite{HMNW}) and in codimension $\leq 3$ every Artinian Gorenetein algebra have unimodal Hilbert vector (see \cite{St3}), on the contrary, in high codimension
the Hilbert vector of a Artinian Gorensten algebra does not need to be unimodal. In \cite{BL}, the authors studied Artinian Gorenstein algebras whose Hilbert vector are non-unimodal, they 
appear in codimension $\geq 5$. 

\par

We recall that the Lefschetz properties for standard graded Artinian $\K$-algebras are algebraic abstractions motivated by the Hard Lefschetz Theorem on the cohomology rings of smooth complex projective varieties; see for instance the survey  
\cite{La} for the theorem and \cite{Ru} for an overview. The Poincaré duality for these cohomology rings inspired the definition of Poincaré duality algebras which, 
in this context, is equivalent to the Gorenstein hypothesis (see \cite{MW} and \cite{Ru}). In \cite{Wa2} and \cite{MW} the authors used Macaulay-Matlis duality in characteristic 
zero to present the Artinian Gorenstein algebra as $A=Q/\ann_Q(f)$ where $f \in R=\K[x_1,\ldots,x_n]$ a polynomial ring and $Q=\K[X_1,\ldots,X_N]$ the associated ring of differential 
operators.  \par

Our strategy to construct standard graded Artinian Gorenstein algebras presented by quadrics is to deal with the simplest ones, which are those whose defining ideal contains the complete intersection 
$(x_1^2,\ldots,x_n^2)$. This assumption forces all monomials that occur in $f$ to be square free. As a matter of fact we deal with bihomogeneous forms of bidegree $(1,d-1)$ of special type, 
here called of monomial square free type (see Definition \ref{defin:bigraded1}). We associate to any bihomogeneous form of monomial square free type, bijectively, 
a pure simplicial complex whose combinatoric structure determines a set of generators of the annihilator ideal, see Theorem \ref{thm:mainideal}. This combinatorial object 
also characterizes the associated algebra presented by quadrics (see Theorem \ref{thm:mainpresentedbyquadrics}). Inspired by the 
famous Turan's graph Theorem (see \cite{Tu}) that characterizes maximal graphs not containing a complete subgraph $K_l$, we introduce a simplicial complex, here called Turan complex, 
whose associated algebra is always presented by quadrics and such that, in very large codimension with respect to the socle degree, the Hilbert vector of $A$ is totally non unimodal, that is 
$\dim A_1 > \dim A_2>\ldots>\dim A_{\lfloor \frac{d}{2}\rfloor} $.  \par

We now describe the contents of the paper in more detail. In the first section we recall the basic definitions and constructions of standard graded Artinian Gorenstein algebras, we deal also with 
the bigraded case which is of particular interest. We recall the Lefschetz properties and Macaulay-Matlis duality. \par

The second section is devoted to the main results and constructions. Theorem \ref{thm:mainideal} describes the annihilator of a standard bigraded form of bidegree $(1,d-1)$ 
of monomial square free type, showing that it is a binomial ideal whose generators are determined by the combinatoric of the associated simplicial complex. It is the main thecnical result. 
 Theorem \ref{thm:mainpresentedbyquadrics} characterizes those algebras are presented by quadrics. 
We introduce the Turan complex in Definition \ref{def:turan} and in Theorem \ref{thm:turanispresentedbyquadrics} we show that Turan algebras are presented by quadrics and determine its Hilbert vector. 
In Corollary \ref{cor:matador} we produce counterexamples to both Migliore-Nagel conjectures in any socle degree $d \geq 4$ and sufficient large codimension. They are Turan algebras of special type.
It is surprising that if the codimension is very large with respect to the socle degree, the Hilbert vector of Turan algebras, that are quadratic and binomial, is totally non-unimodal.
In fact monomial and closely related ideals, in characteristic zero, are expected to have the Weak Lefschetz property (see \cite{MMN}).

%\break

\section{Combinatorics, Lefschetz properties and Macaulay-Matlis duality} 

\subsection{Combinatorics} \label{comb}
 
\begin{defin}\rm
 Let $V = \{u_1,\ldots,u_m\}$ be a finite set. A simplicial complex $\Delta$ with vertex set $V$ is a collection of subsets of $V$, {\it i.e.} a subset of the power set $ 2^V$, such that for all $A \in \Delta$ and for 
 all subset $B \subseteq A$ we have $B \in \Delta$. The members of $\Delta$ are referred as faces and maximal faces (with respect to the inclusion) are the facets. 
 If $A \in \Delta$ and $|A|=k$, it is called a $(k-1)$-face, or a face of dimension $k-1$. 
 If all the facets have the same dimension $d$ the complex is said to be homogeneous of (pure)
 dimension $d$. We say that $\Delta$ is a simplex if  $\Delta = 2^{V}$.
\end{defin}

In our context we identify the faces of a simplicial complex with monomials in the variables $\{u_1,\ldots,u_m\}$. Let $\K$ be any field and let $R = \K[u_1,\ldots,u_m]$ be the polynomial ring. To any finite subset $F \subset \{u_1,\ldots,u_m\}$ 
we associate the monomial $m_F = \displaystyle \prod_{u_i \in F}u_i $. In this way there is a natural bijection between the simplicial complex $\Delta$ and the set of the monomials $m_F,$ where $F$ a facet of $\Delta$.

%\begin{defin}\rm Let $\Delta$ be a simplicial complex over $V$.   
%\end{defin}

\subsection{The Lefschetz Properties}

Let $\K$ be an infinite field and  $R=\K[x_1,\ldots,x_n]$ be the polynomial ring in $n$ indeterminates. 
\begin{defin}\rm
 Let $A$ be a standard graded $\K$-algebra.  We say that $A$ is presented by quadrics if $A\simeq R/I$, where $R=\K[x_1,\ldots,x_n]$ and the homogeneous ideal $I$ 
 has a set of generators consisting of quadratic forms. 
\end{defin}
Let $A=R/I$ be an Artinian standard graded $R$-algebra, then $A$ has a decomposition $A=\displaystyle \bigoplus_{i=0}^dA_i,$ as a sum of finite dimensional $\K$-vector spaces with $A_d\ne 0$.

Let $A=R/I$ be an Artinian standard graded $R$-algebra. 
A form $F\in R_d$ induces a $\K$-vector spaces map $\varphi_{i,F}:A_i\to A_{i+d},$ defined by $\varphi_{i,F}(\alpha)=F\alpha,$ for every $\alpha\in A_i.$

\begin{Definition}\rm
We say that $A$ has the {\em Strong Lefschetz property} (in short SLP) if there exists a linear form $L\in R_1$ such that
$\rk\varphi_{i,L^k}=\min\{\dim_{\K} A_i,\dim_{\K} A_{i+k}\},$ for every $i,k.$ 
%$\bullet L:A_i \to A_{i+1}$ has maximal rank for every $i.$ 
\end{Definition}

\begin{Definition}\rm
We say that $A$ has the {\em Weak Lefschetz property} (in short WLP) if there exists a linear form $L\in R_1$ such that
$\rk\varphi_{i,L}=\min\{\dim_{\K} A_i,\dim_{\K} A_{i+1}\},$ for every $i.$ 
%$\bullet L:A_i \to A_{i+1}$ has maximal rank for every $i.$ 
\end{Definition}

%\begin{Definition}\rm
%We say that $A$ has the {\em Strong Lefschetz property} (in short SLP) if there exists a linear form $L\in R_1$ such that 
%$\rk\varphi_{i,L^d}=\min\{\dim_{\K} A_i,\dim_{\K} A_{i+d}\},$ for every $i$ and $d.$ 
%\end{Definition}

\begin{defin}\rm 
 Let $R=\K[x_1,\ldots,x_n]$ and $A=R/I$ be an Artinian standard graded $R$-algebra, with $I_1=0$. The integer $n$ is said to be the codimension of $A$. 
 If $A_d \neq 0$ and $A_i=0 $ for all $i > d$, then $d$ is called the socle degree of $A$. The Hilbert vector of $A$ is $h_A=\Hilb(A)=(1,h_1,h_2,\ldots,h_d)$, where 
 $h_k = \dim A_k$. We say that $h_A$ is unimodal if there exists $k$ such that $1\leq h_1\leq \ldots \leq h_k \geq h_{k+1}\geq h_d$.
\end{defin}

\begin{rmk}\rm \label{rmk:simplifica}
We recall that an Artinian algebra $A=\displaystyle \bigoplus_{i=0}^dA_i,$ $A_d\ne 0,$ is a {\em Gorenstein algebra} if and only if $\dim_{\K} A_d=1$ and the bilinear pairing 
$$A_i\times A_{d-i}\to A_d$$ inducted by the multiplication is non-degenerated for $0\le i\le d$. 
So we have an isomorphism $A_{i}\simeq\Hom_{\K}(A_{d-i},A_d)$ for $i=0,\ldots,d$.
%In this case, after choose an isomorphism $A_d \simeq \K$, we have $A_{d-i} \simeq A_i^*$ for $i=0,\ldots,d$. 
In particular, $\dim_{\K} A_i=\dim_{\K} A_{d-i},$ for $i=0,\ldots,d$.
Moreover, for every $L\in R_1,$ $\rank\varphi_{i,L}=\rank\varphi_{d-i-1,L},$ for $0\le i\le d.$
\par
Since $A$ is generated in degree $0$ as a $R$-module,  if $\varphi_{i,L}$ is surjective, then $\varphi_{j,L}$ is surjective for every $j\ge i.$ 
Therefore, if $A$ is a Gorenstein Artinian algebra, if $\varphi_{i,L}$ is injective, then $\varphi_{j,L}$ is injective for every $j\le i.$ 
Of course SLP implies WLP. Notice also that the WLP implies the unimodality of the Hilbert vector of $A$. Unimodality in the Gorenstein case implies that $\dim A_{k-1} \leq \dim A_k$ for all $k \leq \frac{d}{2}$.
The converse of these implications are not true, (see \cite{Go}).
\end{rmk}

%\par\rosso{inserire l'osservazione che basta un sol posto per WLP}
\par

\subsection{Macaulay-Matlis duality}

Now we assume that $\chr\K=0.$ Let us regard the polynomial algebra $R$ as a module over the algebra $Q=\K[X_1, . . . ,X_n]$ via the identification $X_i = \partial/\partial x_i.$ If $f\in R$ we set
 $$\ann_Q(f)=\{p(X_1,\ldots,X_n)\in Q\mid p(\partial/\partial x_1,\ldots,\partial/\partial x_n)f=0\}.$$
 By Macaulay-Matlis duality we have a bijection:

$$\begin{array}{ccc}
\{\text{Homogeneous ideals of} \ R\} & \leftrightarrow & \{\text{Graded}\ R-\text{submodules}\ \text{of}\ Q \}\\
\operatorname{Ann}_Q(M) & \leftarrow & M\\
I & \rightarrow & I^{-1}
  \end{array}
$$
 
Let $I\subset Q$ be a homogeneous ideal. It is well known that $A=Q/I$ is a Gorenstein standard graded Artinian algebra if and only if there exists a form $f\in R$ such that $I=\ann_Q(f)$ (for more details see, for instance, \cite{MW}).

In the sequel we always assume that $\operatorname{char}(\K)=0$, $A=Q/I$, $I=\ann_Q(f)$ and $I_1=0$. When we need to assume that $\K$ is algebraically closed it will be explicit. All arguments work over $\C$.

We deal with standard bigraded Artinian Gorenstein algebras $A=\displaystyle \bigoplus_{i=0}^d A_i,$ $A_d\ne 0$, 
with $A_k=\displaystyle \bigoplus_{i=0}^k A_{(i,k-1)}$, $A_{(d_1,d_2)}\ne 0$ for some $d_1,d_2$ such that $d_1+d_2=d$, we call $(d_1,d_2)$ the socle bidegree of $A$. Since $A^*_k \simeq A_{d-k} $ and since duality is compatible with 
direct sum, we get $A_{(i,j)}^* \simeq A_{(d_1-i,d_2-j)}$. \\
Let $R=\K[x_1,\ldots,x_n,u_1,\ldots,u_m]$ be the polynomial ring viewed as standard bigraded ring in the sets of variables $\{x_1,\ldots,x_n\}$ and $\{u_1,\ldots,u_m\}$
and let $Q=\K[X_1,\ldots,X_n,U_1,\ldots,U_m]$ be the associated ring of differential operators. 

We want to stress that the bijection given by Macaulay-Matlis duality preserves bigrading, that is, there is a bijection:

$$\begin{array}{ccc}
\{\text{Bihomogeneous ideals of} \ R\} & \leftrightarrow & \{\text{Bigraded}\ R-\text{submodules}\ \text{of}\ Q \}\\
\operatorname{Ann}_Q(M) & \leftarrow & M\\
I & \rightarrow & I^{-1}
  \end{array}
$$

 If $f \in R_{(d_1,d_2)}$ is a bihomogeneous polynomial of total degree $d=d_1+d_2$, 
then $I = \ann_Q(f) \subset Q$ is a bihomogeneous ideal and $A = Q/I$ is a standard bigraded Artinian Gorenstein algebra of socle bidegree $(d_1,d_2)$ and codimension $r=m+n$ if we assume, without lost of generality, that $I_1=0$.

\begin{rmk} \label{rmk:bigradedideal}\rm  If $f \in R_{(d_1,d_2)}$ is a bihomogeneous polynomial of bidegree $(d_1,d_2)$, consider the associated bigraded algebra $A$ of socle bidegree $(d_1,d_2)$.  
 Notice that for all $\alpha \in Q_{(i,j)}$ with $i>d_1$ or $j>d_2$ we get $\alpha(f)=0$, therefore, under these conditions $I_{(i,j)}=Q_{(i,j)}$. As consequence, we have the following decomposition for all $A_k$:
 $$A_k = \displaystyle \bigoplus_{i+j=k,i\leq d_1,j\leq d_2} A_{(i,j)}.$$
Furthermore, for $i<d_1$ and $j<d_1$, the evaluation map $Q_{(i,j)} \to A_{(d_1-i,d_2-j)}$ given by $\alpha \mapsto \alpha(f)$ provides the following short exact sequence:
$$0 \to I_{(i,j)} \to Q_{(i,j)} \to A_{(d_1-i,d_2-j)} \to 0.$$
\end{rmk}

One of our goals is to produce bigraded algebras of socle bidegree $(1,d-1)$ presented by quadrics. In order to achieve this objective we study the ideal of a particular family.

\begin{defin}\rm \label{defin:bigraded1}
 With the previous  notation, all bihomogeneous polynomials of bidegree $(1,d-1)$ can be written in the form
 $$f= x_1g_1+\ldots+x_ng_n,$$
 where $g_i \in \K[u_1,\ldots,u_m]_{d-1}$. We say that $f$ is of {\em monomial square free type} if all 
 $g_i$ are square free monomials. The associated algebra, $A = Q/\ann_Q(f)$, is bigraded, has socle bidegree $(1,d-1)$ and we assume that $I_1=0,$ so $\codim A = m+n$.
\end{defin}

The combinatoric structure inward bihomogeneous polynomials of monomial square free type allows us to give necessary and sufficient 
conditions in order to the associated algebra to be presented by quadrics. On the other hand we construct, 
in sufficiently large codimension, Artinian Gorenstein algebras presented by quadrics failing the WLP.

\break

\section{The main results}

Let $\Delta$ be a homogeneous simplicial complex of dimension $d-2$ whose facets are given by the monomials $g_i \in \K[u_1,\ldots,u_m]_{d-1}$ (see $\S$ \ref{comb}). Let $f  \in \K[x_1,\ldots,x_n,u_1,\ldots,u_m]_{(1,d-1)}$ be the bihomogeneous form of monomial square free type
associated to $\Delta$, that is $f=f_{\Delta}=\displaystyle \sum_{i=1}^nx_ig_i$ (see Difinition \ref{defin:bigraded1}).  The vertex set of $\Delta$ is also called $0$-skeleton and and we write $V=\{u_1,\ldots,u_m\}$. We identify the $1$-skeleton with 
a simple graph $\Delta_1=(V,E)$, hence the $1$-faces are called edges. Since, by differentiation, $X_i(f)=g_i$, we can identify each facet $g_i$ with the differential operator $X_i$. We denote by $e_k$ the number of $(k-1)$-faces, hence $e_1=m$ and $e_{d-1}=n$ and we put 
$e_0:=1$ and $e_j:=0$ for $j \geq d-1$. 
Let $A=Q/\ann_Q(f)$ be the associated algebra, we suppose that $I_1=0$. We identify the faces of $\Delta$ with the dual differential operators by $u_i \leftrightarrow U_i$. 
If $p \in \K[u_1,\ldots,u_m]$ is a square free monomial, we denote by $P \in \K[U_1,\ldots,U_m]$ the dual differential operator $P=p(U_1,\ldots,U_m)$. Notice that $P(p)=1$.

\begin{defin}
 Let $\Delta$ be a homogeneous simplicial complex of dimension $d-2$, the associated algebra is $A_{\Delta} = Q/\ann(f_{\Delta})$.
\end{defin}

\begin{thm}\label{thm:mainideal}
 Let $\Delta$ be a homogeneous simplicial complex of dimension $d-2$ and let $A_{\Delta}$ be the associated algebra. Then \begin{enumerate}
                                                                                      \item $A=\displaystyle \bigoplus_{k=0}^d A_k$ where $A_k = A_{(0,k)} \oplus A_{(1,k-1)}$.
                                                                                     \item $A_{(0,k)}$ has a basis identified with the $(k-1)$-faces of $\Delta$, hence $\dim A_{(0,k)} = e_k$.
                                                                                     \item By duality, $A_{(1,k-1)}^* \simeq A_{(0,d-k)}$, and a basis for $A_{(1,k-1)}$ can be chosen 
by taking, for each $(d-k-1)$-face of $\Delta$, a monomial $X_i\tilde{G}_i$ such that $X_i\tilde{G}_i(f)$ represents it.
\item The Hilbert vector of $A$ is given by $h_k = \dim A_k = e_k+e_{d-k}$.
\item  Furthermore, $I=\ann_Q(f)$ is generated by 
\begin{enumerate}
 \item $(X_1,\ldots,X_n)^2;$ $U_1^2,\ldots,U_m^2;$
 \item The monomials in $I$ representing minimal non faces of $\Delta;$
 \item The monomials $X_iF_i$ where $f_i$ does not represent a subface of $g_i;$
 \item The binomials $X_i\tilde{G}_i-X_j\tilde{G}_j$ where 
$g_i=\tilde{g}_ig_{ij}$ and $g_j=\tilde{g}_jg_{ij}$ and $g_{ij}$ represents a common subface of $g_i,g_j$. 

\end{enumerate}

                                                                                     \end{enumerate}

 \end{thm}

\begin{proof}
 It is easy to see that $A_{(0,k)}$ is generated by the monomials of degree $k$ that represent $(k-1)$-faces, since they are the only ones that do not annihilate $f$. Now we show that they are linearly independent over $\K$. For any $(k-1)$-face 
 $\omega$, let $\Omega$ be the associated monomial of $Q_{(0,k)}$, let $\Omega_1,\ldots,\Omega_s$ be all of them. Since $\Omega(f) = \displaystyle \sum_{i=1}^nx_i\Omega(g_i)$, if we take any  linear 
 combination $$0=\displaystyle \sum_{j=1}^s c_j\Omega_j(f) = \displaystyle \sum_{j=1}^s c_j \displaystyle \sum_{i=1}^nx_i\Omega_j(g_i)= \displaystyle \sum_{i=1}^n x_i \displaystyle \sum_{j=1}^s c_j \Omega_j(g_i).$$
We get $\displaystyle \sum_{j=1}^s c_j \Omega_j(g_i)=0$ for all $i=1,\ldots,n$. For a fixed $i$, $\Omega_j(g_i)$ are distinct monomials or zero, but for each $j$ there is a $i$ such that 
$\Omega_j(g_i)\neq 0$, therefore $c_j=0$ for all $j =1,\ldots,s$. The other assertions about $A$ are now clear. \\ 
 Notice that $I_{(i,j)}= Q_{(i,j)}$ for all $i \geq 2$ and it is generated by $I_{(2.0)} = (X_1,\ldots,X_n)^2$. Now we describe 
 $I_{(0,k)}$ and $I_{(1,k-1)}$. Consider the exact sequence given by evaluation: 
 $$0 \to I_{(0,k)} \to Q_{(0,k)} \to A_{(1,d-1-k)}\to 0.$$
 Since $\dim A_{(1,d-1-k)} = \dim A_{(1,d-1-k)}^* = \dim A_{(0,k)} = e_k$, we get $\dim I_{(0,k)} = \dim Q_{(0,k)} - e_k$. Since $\dim A_{(0,k)} =e_k $ and it has a basis given 
 by the $(k-1)$-faces of $\Delta$ and since all the other $\dim Q_{(0,k)} - e_k$ monomials are linearly independent elements of $I_{(0,k)}$, they form a basis for it. Consider the sequence given by evaluation: 
 $$0 \to I_{(1,k-1)} \to Q_{(1,k-1)} \to A_{(0,d-k)}\to 0.$$
We have $\dim I_{(1,k-1)} = \dim Q_{(1,k-1)} - e_{d-k}$. Let us write $Q_{(1,k-1)} = \bar{I}_{(1,k-1)}\oplus \tilde{Q}_{(1,k-1)}$ where $\bar{I}_{(1,k-1)} $ is the $\K$-vector space spanned 
by the monomials $X_iF_i$ where $F_i$ does not represent a subface of $G_i$. Of course $\bar{I}_{(1,k-1)}\subset I_{(1,k-1)}$ and $\tilde{Q}_{(1,k-1)}$ is spanned by all the monomials 
$X_i\tilde{G}_i$ where $\tilde{G}_i$ is a subface of $G_i$. The exact sequence given by evaluation restricted to $\tilde{Q}_{(1,k-1)}$ becomes
$$0 \to \tilde{I}_{(1,k-1)} \to \tilde{Q}_{(1,k-1)} \to A_{(0,d-k)}\to 0.$$
Hence, $I_{(1,k-1)} = \tilde{I}_{(1,k-1)}\oplus \bar{I}_{(1,k-1)}$, since $X_i\tilde{G}_i(f)$ is a face of $\Delta$, $\tilde{I}_{(1,k-1)}$ is generated by the binomials 
$X_i\tilde{G}_i-X_j\tilde{G}_j$ such that $X_i\tilde{G}_i(f)= g_{ij}=X_j\tilde{G}_j(f)$ where $g_{ij}$ is a common subface of $g_i,g_j$, $g_i=\tilde{g_i}g_{ij}$ and $g_j=\tilde{g_j}g_{ij}$. 
The result follows.
 \end{proof}

 \begin{defin}\rm
  Let $\Delta$ be a homogeneous simplicial complex of dimension $d-2$. We say that $\Delta $ is facet connected if for any pair of facets $F,F'$ of $\Delta$ there exists a sequence of facets, 
  $F_0=F,F_1,\ldots,F_s=F'$ such that $F_i \cap F_{i+1}$ is a $(d-3)$-face. We say that $\Delta$ is a flag complex if every collection of pairwise adjacent vertices spans a simplex.
   \end{defin}
 
 \begin{rmk}\rm 
 The difinition of a flag complex $\Delta$ is equivalent to say that for all complete subgraphs $H=K_l \subset \Delta_1$ for $l \geq 3$,
  there exists a $(l-1)$-face $F \in \Delta_l$ such that $H$ is the first skeleton of $F$.  
 In particular, if $\Delta$ is a flag complex, then $\Delta_1$ does not contain any $K_{d-1}$.
  
 \end{rmk}

\begin{thm}\label{thm:mainpresentedbyquadrics}
Let $\Delta$ be a homogeneous simplicial complex of dimension $d-2\geq 1$ and let $A_{\Delta}$ be the associated algebra. $A$ is presented by quadrics if and only if $\Delta$ is a facet connected flag complex.
\end{thm}

\begin{proof} Suppose that $\Delta$ is a facet connected flag complex and let $I=\ann_Q(f_{\Delta})$. By applying Theorem \ref{thm:mainideal} to $I,$ it is enough to consider the monomials in the $U_i$ that does not represent a face of $\Delta$, monomials 
 $X_iF_i$ where $F_i$ is a monomial in the $U_j$ that does not represent a subface of $G_i$ and the binomials 
$X_i\tilde{G}_i-X_j\tilde{G}_j$ where $X_i\tilde{G}_i(f)= g_{ij}=X_j\tilde{G}_j(f)$ is a common subface of $g_i,g_j$. 
Let $M = U_1^{e_1}\ldots U_m^{e_m}$ be a monomial such that $M(f)=0$, since $U_i^2 \in I$ we can consider $M$ square free and suppose that it does not represent a face of $\Delta$. In this case, the first skeleton of $M$ represents a complete graph $K_l$ with the same vertex set of $G$, since, by hypothesis, 
for $3 \leq l \leq d-2$ all $K_l \subset G$ comes from a $l$-face of $\Delta$ and since $G$ does not contain a $K_{d-1}$ as subgraph, 
there exists $U_iU_j$ in $M$ such that $U_iU_j(f)=0$ and $M=U_iU_j\tilde{M}\in I_2Q$.\par
Let $\Omega = X_iM$ with $M = U_1^{e_1}\ldots U_m^{e_m}$ be a monomial such that $\Omega(f)=M(g_i)=0$, we can suppose that $M$ is square free and it does not represent a subface of $g_i$, hence there is a $U_j$ in $M$ that does not belongs to $G_i$, yielding $\Omega = X_iU_j\tilde{M} \in I_2Q$. \\
To finish the proof, consider the binomials $X_i\tilde{G}_i-X_j\tilde{G}_j$ where $X_i\tilde{G}_i(f)= g_{ij}=X_j\tilde{G}_j(f)$ and $g_{ij}$ is a common subface of $g_i,g_j$. 
If $\tilde{G_i}$ and $\tilde{G_j}$ are subfaces of the facets $G_i,G_j$ respectively and if $g_{ij} \subset G_i\cap G_j$ and the intersection is a $(d-3)$-face, then 
there are only two vertexes they do not share, say $u_i,u_j$, $\tilde{G}_i=U_iG_{ij}$ and $\tilde{G}_j=U_jG_{ij}$ and finally $X_i\tilde{G}_i-X_j\tilde{G}_j=(X_iU_i-X_jU_j)G_{ij}\in I_2Q$. 
In the general case, by the facet connection of $\Delta$, there exists a sequence of facets $G_{i_0}=G_i,G_{i_1},\ldots,G_{i_s}=G_j$ such that the intersection of two consecutive 
facets is a $(d-3)$-face, hence $X_i\tilde{G}_i-X_{i_1}\tilde{G}_{i_1},X_{i_1}\tilde{G}_{i_1}-X_{i_2}\tilde{G}_{i_2},\ldots,X_{i_s}\tilde{G}_{i_s}-X_j\tilde{G}_j\in I_2Q$, summing up 
we get the desired result.\par

Conversely, if $\Delta$ is not facet connected, let $g_j,g_j$ be two facets that can not be facet connected and let $g_{ij} = \gcd(g_ig_j)$. 
By Theorem \ref{thm:mainideal} it is easy to see that $X_i\tilde{G}_i-X_j\tilde{G}_j$ is a minimal 
generator of $I$ where $g_i=\tilde{g}_ig_{ij}$ and $g_j=\tilde{g}_jg_{ij}$. If $\Delta$ is not a flag complex, then there is a complete subgraph $K_s \subset G$ that does not came from a $s$-face of $\Delta$. In this case, if we choose $s$ to be minimal, then 
by Theorem \ref{thm:mainideal} the monomial $M=\displaystyle \prod_{v \in V(K_s)} v$ is a minimal generator of $I$.\par

\end{proof}

We introduce the following complexes inspired by the famous Turan's Graph Theorem characterizing maximal graphs not containing a complete subgraph $K_{d-1}$ as the $(d-2)$-partite complete graph 
$K(a_1,\ldots,a_{d-1})$ with $|a_i-a_j|\leq 1$ (c.f. \cite{Tu}). 

\begin{defin}\rm \label{def:turan} Let $2 \leq a_1 \leq \ldots \leq a_{d-1}$ be integers. The Turan complex of order $a_1,\ldots,a_{d-1}$, $\mathcal{K}=\mathcal{TK}(a_1,\ldots,a_{d-1})$, is the homogeneous simplicial complex whose facets set is the 
cartesian product $\pi = \displaystyle \prod_{i=1}^{d-1}\{1,2,\ldots,a_i\}$. The associated algebra is called the Turan algebra of order $(a_1,\ldots,a_{d-1})$ and denoted by $TA(a_1,\ldots,a_{d-1})$.
\end{defin}

\begin{thm} \label{thm:turanispresentedbyquadrics}
 Every Turan algebra $TA(a_1,\ldots,a_{d-1})$ is presented by quadrics. Its Hilbert vector is given by $h_k = s_{k}+s_{d-k}$ where $s_k = s_k(a_1,\ldots,a_{d-1})$ is the elementary symmetric polynomial of order $k$. 
\end{thm}

\begin{proof}
 By Theorem \ref{thm:mainpresentedbyquadrics}, the first claim is equivalent to prove that every Turan complex is a facet connected flag complex. 
 Let $2 \leq a_1 \leq \ldots \leq a_{d-1}$ be integers and consider the Turan complex $\mathcal{K}=\mathcal{TK}(a_1,\ldots,a_{d-1})$.\\
 To show that $\mathcal{K}$ is facet connected, let us consider $F,F'$ two of its facets. $F = \{x_1,\ldots,x_{d-1}\}$ and $F' = \{y_1,\ldots,y_{d-1}\}$ with $x_i,y_i \in \{1,\ldots,a_i\}$. 
 Consider the following sequence of facets in $\mathcal{K}$: $F_0=F$, $F_1 = (F \cup y_1) \setminus x_1$, we have that $F_0 \cap F_1$ is a $(d-3)$ face; and we construct inductively, for $k \in {1,\ldots,d-1}$, 
 $F_{k} = (F_{k-1} \cup y_k) \setminus x_k$. it is easy to see that $F_k \cap F_{k-1}$ is a $(d-3)$-face and that $F_{d-1} = F'$, therefore, $\mathcal{K}$ is facet connected as claimed. \\
 To show that $\mathcal{K}$ is a flag complex, first notice that $\mathcal{K}$ does not contain a complete graph $K_{d-1}$ in its first skeleton, by the $d-1$-coloration. 
 Let us consider any complete subgraph of the first skeleton $H = K_l \subset \mathcal{K}_1$ with $3 \leq l \leq d-2$. We can suppose without loss of generality that the vertex set of $H$ is $V = \{x_1,\ldots,x_l\}$ with 
 $x_i \leq a_i$. By definition of $\mathcal{K}$, there is a a facet of $\mathcal{K}$ whose vertex set contains $V$, by the definition of simplicial complex, there is a face of $\mathcal{K}$ such that the first skeleton is $H$ and the result follows.  \\
 The second claim follows from the fact that the number of $(k-1)$-faces of a Turan complex is $e_k=s_k$ where $s_k=s_k(a_1,\ldots,a_{d-1})$ is the elementary symmetric polynomial of order $k$.   
 By Theorem \ref{thm:mainpresentedbyquadrics}, the Hilbert vector of the Turan algebra $TA(a_1,\ldots,a_{d-1})$ is given by $h_k=s_k+s_{d-k}$.

\end{proof}

We now present a family of counterexamples to Migliore-Nagel conjectures that occur in large codimension with respect to the socle degree. 
%Among these counterexamples there is a subclass such that the Hilbert vector becomes non unimodal in the first step.

\begin{cor}\label{cor:matador}
 Let $A=TA(a_1,\ldots,a_{d-1})$ be the Turan algebra of order $(a_1,\ldots,a_{d-1})$ with $a_1 \approx \ldots \approx a_{d-1}$ large enough. Then $\Hilb(A)$ is totally non-unimodal, 
 that is $$\dim A_1 > \dim A_2>\ldots>\dim A_{\lfloor \frac{d}{2} \rfloor}.$$
 \end{cor}

\begin{proof} 
If $a_1 \approx \ldots \approx a_{d-1} \approx a$ are large enough, then, by a trivial Calculus I argument, we get for $2\leq k+1 \leq \lfloor \frac{d}{2} \rfloor$, $k < d-k$ and $d-k> d-k-1 \geq k+1$:

  $$\dim A_k \approx \binom{d-1}{k}a^k+\binom{d-1}{d-k}a^{d-k}>\binom{d-1}{k+1}a^{k+1}+\binom{d-1}{d-k-1}a^{d-k-1} \approx \dim A_{k+1}.$$  In this case, 
  the Hilbert vector $\Hilb(A)$ is totally non-unimodal.
 \end{proof}

{\bf Acknowledgments}.
We wish to thank Francesco Russo for his suggestions and conversations on the subject and the participants to the Commutative Algebra/Algebraic Geometry Seminar of the 
{\it Università di Catania}. The first author thanks the Sicilian hospitality he found everywhere in Catania, rendering his long staying there very pleasant and fruitful. The first author also thanks 
the participants of BIRS-workshop on Lefschetz porperties and artinian algebras for inspiring conversations on the subject. 
The first author was Partially supported  by the CAPES postdoctoral fellowship, Proc. BEX 2036/14-2. 
The second author is part of the Research Project of the University of Catania FIR 2014 ``Aspetti geometrici e algebrici della Weak e Strong Lefschetz Property".

\end{document}